\documentclass{article}
\usepackage{amssymb}
\usepackage{graphicx}    
\usepackage{color}
\newcommand{\bC}{\mathbf{C}}
\newcommand{\bR}{\mathbf{R}}
\newcommand{\ord}{\mbox{\rm ord}}
\newcommand{\Nj}{{\cal N}_J}
\newcommand{\jac}{\mathop{\mathrm{jac}}}
\newcommand{\unit}{\hbox{$\,o\!\!\!\!$\lower0.06ex\hbox{---}\,}}

\newtheorem{Theorem}{Theorem}[section]
\newtheorem{Example}[Theorem]{Example}
\newtheorem{Lemma}[Theorem]{Lemma}
\newtheorem{Definition}[Theorem]{Definition}

\newcommand{\Teis}[2]{\left\{
   \setlength{\unitlength}{1ex}
   \begin{picture}(2,0)(0,0.4)
      \put(0,1.1){\line(1,0){2}}
      \put(0,0.9){\line(1,0){2}}
      \put(1,1.2){\makebox(0,0)[b]{$\scriptstyle #1$}}
      \put(1,0.8){\makebox(0,0)[t]{$\scriptstyle #2$}}
   \end{picture}\right\}}

\newcommand{\Teisssr}[4]{\left\{
   \setlength{\unitlength}{1ex}
   \begin{picture}(#3,3)(0,0.4)
      \put(0,1.15){\line(1,0){#3}}
      \put(0,0.85){\line(1,0){#3}}
      \put(#4,1.3){\makebox(0,0)[b]{$#1$}}
      \put(#4,0.7){\makebox(0,0)[t]{$#2$}}
   \end{picture}\right\}}

\newenvironment{proof}[1][Proof]{\textbf{#1.} }{\
\rule{0.5em}{0.5em}}

\author{Janusz Gwo\'zdziewicz\\
\texttt{matjg@tu.kielce.pl}
}
\title{Invariance of the jacobian Newton diagram}
\begin{document}
\maketitle

\begin{abstract}
We prove that the jacobian Newton diagram of the holomorphic mapping $(f,g):(\bC^2,0)\to(\bC^2,0)$
depends only on the equisingularity class of the pair of curves $f=0$ and $g=0$. 
\end{abstract}

\section{Introduction}\label{Intro}
Write $\bR_{+}=\{\,x\in\bR:x\geq0\,\}$. The Newton diagram $\Delta_h$ of a power series 
$h(x,y)=\sum_{i,j}c_{ij}x^iy^j$ is by definition the convex hull of the union 
$$ \bigcup_{ \{ (i,j) : c_{ij}\neq0 \} }  \!\! \{\,(i,j)+\bR_{+}^2\,\}. $$

\noindent
\parbox{3in}{\begin{Example}\label{Ex:ND}
The Newton diagram of $h(x,y)=y^5+2xy^3-x^3y^2+3x^4y$ is drawn in the figure.
Black dots are the points of the first quadrant $\bR_{+}^2$ corresponding to  
non-zero monomials of the series~$h$.
\end{Example}}
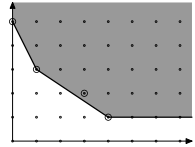
\hspace{5em}\includegraphics[scale=0.6, trim=10mm 10mm 60mm 60mm]{maciek.1}    

Let $\phi:(\bC^2,0)\to(\bC^2,0)$, $\phi^{-1}(0,0)=\{(0,0)\} $
 be a germ of a holomorphic  mapping given by $(x,y)=(f(u,v),g(u,v))$. 
Let $\jac\phi=\frac{\partial f}{\partial u}\frac{\partial g}{\partial v} - 
\frac{\partial f}{\partial v}\frac{\partial g}{\partial u}$ 
be the usual jacobian determinant. 
The direct image of the curve germ $\jac\phi=0$ by~$\phi$ is called the {\em
discriminant curve} of $\phi$ (see \cite{Casas}). 
If $D(x,y)=0$ is an analytic equation of the discriminant
curve then the Newton diagram of $D$ is called the
{\em jacobian Newton diagram} of $(f,g)$. 
We will write $\Nj(f,g)$ for the jacobian Newton diagram.

\begin{Definition}
Let $\xi$, $\xi'$, $\nu$, $\nu'$ be germs of analytic curves in $(\bC^2,0)$. 
We say that the pairs of curves $\xi$, $\nu$ and $\xi'$, $\nu'$  are equisingular 
if there exists a homeomorphism $\Psi:(\bC^2,0)\to(\bC^2,0)$
preserving the multiplicity of each branch such that $\Psi(\xi)=\xi'$ and
$\Psi(\nu)=\nu'$.
\end{Definition}

\section{Main result}\label{Intro}
\begin{Theorem}\label{Th:2}
Let $(f,g):(\bC^2,0)\to(\bC^2,0)$, $(f,g)^{-1}(0,0)=\{(0,0)\}$ 
be a germ of a holomorphic mapping.
Then the jacobian Newton diagram $\Nj(f,g)$ depends only 
on the equisingularity class of the pair of curves $f=0$ and $g=0$.
\end{Theorem}

The proof  is in the last section. 

We give a short survey of results related with Theorem~\ref{Th:2}. 
We need a few notions which will be used only in this section 
to explain connection between certain analytic factorizations of $\jac(f,g)$ 
and the jacobian Newton diagram~$\Nj(f,g)$.

The Minkowski sum of Newton diagrams $\Delta_1$ and $\Delta_2$ is by definition 
$\Delta_1+\Delta_2=\{ p+q: p\in\Delta_1, q\in\Delta_2\,\}$.
The set of Newton diagrams is a semi-group with respect to Minkowski sum and the generators 
of this semi-group are elementary Newton diagrams illustrated in Figure~1. 
\begin{figure}[htb!]
\centering%
\includegraphics{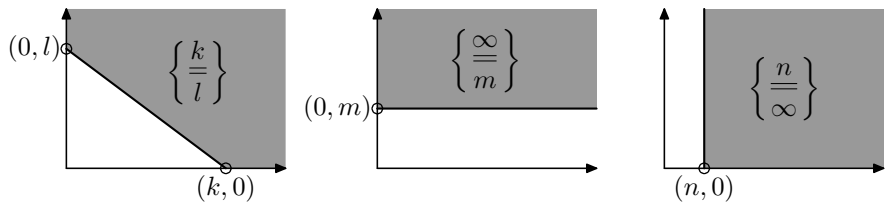}
\caption{Elementary Newton diagrams}
\label{fig:FigureExample}
\end{figure}

The inclination of the elementary Newton diagram~$\Teis{a}{b}$
is the quotient $\frac{a}{b}$ with conventions 
$\frac{\infty}{b}=\infty$ and $\frac{a}{\infty}=0$.  
For an arbitrary Newton diagram $\Delta$ represented as a sum of elementary Newton diagrams 
let us denote $I(\Delta)$  the set of inclinations of elementary Newton diagrams of the sum. 
It is easy to see that $I(\Delta)$ does not depend on the choice of representation. 
Coming back to Example~\ref{Ex:ND} the Newton diagram $\Delta_h$ is the sum 
$\Teis{1}{2}+\Teis{3}{2}+\Teis{\infty}{1}$ 
and $I(\Delta_h)=\{\,1/2,3/2,\infty\,\}$.

For every irreducible factor $h$ of $\jac(f,g)$ the Hironaka number $q(h)=\frac{i_0(g,h)}{i_0(f,h)}$, 
where $i_0(\cdot,\cdot)$ stands for the intersection multiplicity, 
is called the jacobian quotient or the jacobian invariant of $(f,g)$.  

\begin{Definition}
Let $\jac(f,g)=J_1\cdots J_n$ be an analytic factorization of the jacobian. 

We will call $J_1\cdots J_n$ a Hironaka factorization 
if for every~$J_i$ ($1\leq i\leq n)$ the Hironaka number~$q(h)$  
is constant for all irreducible factors $h$ of $J_i$. 

The Hironaka factorization $J_1\cdots J_n$ will be called minimal 
if Hironaka numbers of irreducible factors 
of $J_l$ and $J_k$ are different for $1\leq l<k\leq n$.
\end{Definition}
 
Let $\jac(f,g)=h_1\cdots h_n$ be the factorization of the jacobian into irreducible factors. 
It is easy to check (cf.~\cite{Teissier3}) that
$$\Nj(f,g)=\sum_{i=1}^n\Teisssr{i_0(g,h_i)}{i_0(f,h_i)}{8}{4} .
$$
It follows directly from the above formula that
\begin{itemize}
\item the set of jacobian quotients of $(f,g)$ is the set of inclinations of $\Nj(f,g)$,
\item if $J_1\cdots J_r$ is a Hironaka factorization of $\jac(f,g)$ then 
$$\Nj(f,g)=\sum_{i=1}^r\Teisssr{i_0(g,J_i)}{i_0(f,J_i)}{8}{4} ,
$$     
\item if $\Nj(f,g)=\sum_{i=1}^s\Teis{a_i}{b_i}$ with inclinations $\frac{a_i}{b_i}$  
         pairwise different  then $\jac(f,g)$ has the minimal Hironaka factorization $J_1\cdots J_s$ 
         such that  $i_0(g,J_i)=a_i$ and~$i_0(f,J_i)=b_i$ for $i=1,\dots,s$.
\end{itemize}

Take a germ of a holomorphic mapping $(l,f):(\bC^2,0)\to(\bC^2,0)$
such that $f=0$ is a curve germ without multiple components and $l=0$ is a smooth curve. 
Under these assumptions $\jac(l,f)=0$ is called the polar curve of $f$ with respect to $l$ and 
jacobian quotients of $(l,f)$ are called polar quotients.  
A survey of recent results concerning polar curves  is in~\cite{GLP}.

In \cite{KL} the authors described the contact orders of Newton-Puiseux roots of $f_x'(x,y)=0$ 
with the Newton-Puiseux roots of $f(x,y)=0$. They constructed the tree model~$T(f)$ which encodes 
these contact orders. Using Kuo-Lu tree~$T(f)$ one can compute the set of polar quotients of $(y,f)$. 
One can also give a formula for the jacobian Newton diagram of $(y,f)$ in terms of $T(f)$
(see the last line before Example~5.2 in \cite{GG}). 

Merle in \cite{Merle} obtained the minimal Hironaka decomposition of the polar curve of 
the irreducible curve germ $f=0$ with respect to a smooth curve $l=0$ transverse to $f=0$.
Merle's results  is rewritten in \cite{Teissier3}
as a formula for the jacobian  Newton diagram of $(l,f)$ (see also \cite{GLP}, Theorem~4.1).

In \cite{Egg} the author found a Hironaka factorization of the polar curve of 
a many-branched curve $f=0$.  
He associated the factors with vertexes of a new type of tree $E(f)$ called now Eggers tree.  
Eggers found also the intersection multiplicities of every factor with $l$ and $f$.  
Since the Eggers tree $E(f)$ depends only on the equisingularity class of $(l,f)$, 
Theorem~\ref{Th:2} in this particular case follows from~\cite{Egg}. 

\medskip
Consider now a general case of a holomorphic mapping germ~ 
$(f,g):(\bC^2,0)\to(\bC^2,0)$, where $(f,g)^{-1}(0,0)=\{(0,0)\}$.

In \cite{Kuo-Pa} the authors additionally assumed that the curve $fg=0$ 
has no multiple components.  They introduced the Eggers tree model~$E(f,g)$ 
of a pair~$(f,g)$. They obtained the Hironaka factorization of the jacobian 
associated with vertexes of~$E(f,g)$. However they did not compute 
the intersection multiplicities of some factors (factors associated 
with collinear bars in  terminology of \cite{Kuo-Pa}) with $f$ and $g$. Hence 
Theorem~4.1 does not follow from~\cite{Kuo-Pa}.

In \cite{Ma} and \cite{Michel} the authors resolved singularities  of the curve $fg=0$.
Then they distinguished some subsets of the exceptional divisor called rupture zones 
and  associated with every rupture zone a factor of the jacobian $\jac(f,g)$. 
 Maugendre found in~\cite{Ma} using topological methods  the set of jacobian quotients 
(see also \cite{Casas1} for an algebraic proof)  and Michel completed the work 
computing intersection multiplicity of every factor with $f$ and $g$. Since the decomposition 
of the jacobian obtained by Michel is a Hironaka factorization, Theorem~\ref{Th:2} follows 
from~\cite{Michel}.  However my proof is much simpler as it uses only Theorem~\ref{Th:1} 
and Theorem~2.1 of \cite{Casas}.

\section{Invariance of a generic curve of the pencil}\label{Main}
\begin{Theorem} \label{Th:1}
Let $(f,g):(\bC^2,0)\to(\bC^2,0)$, $(f,g)^{-1}(0,0)=\{(0,0)\}$ 
 be a germ of a holomorphic mapping.
Then for all $t\in\bC$ but a finite number the equisingularity class of
the curve $f(x,y)-tg(x,y)=0$ depends only on the equisingularity 
class of the pair of curves $f=0$ and $g=0$.  
\end{Theorem}

I guess that the above theorem is a known result. 
However I did not find any reference and I decided to prove it. 

\begin{proof}
Our main reference is Chapter III of \cite{Oka}.
Let $R:M\to(\bC^2,0)$ be the minimal good resolution of singularities of the curve $fg=0$. 
The set $R^{-1}(\{fg=0\})$ can be written as the union of irreducible components
$E_1\cup\dots\cup E_n\cup E_{n+1}\cup\dots\cup E_m$,
where $E=E_1\cup\dots\cup E_n$ is the exceptional divisor $R^{-1}(0)$ 
and $E_{n+1}$, \dots, $E_m$  are non-compact curves 
corresponding with branches of the curve $fg=0$.
Put $\tilde f=f\circ R$,  $\tilde g=g\circ R$ and let 
$a_i=\mbox{order of $\tilde f$ along } E_i$, 
$b_i=\mbox{order of $\tilde g$ along } E_i$ for $i=1,\dots, m$.
Then, after renumbering $E_1$, \dots , $E_m$ if necessary, 
the total dual resolution graph as well as the numbers $a_i$ and $b_i$ for~$i=1,\dots,m$ 
depend only on the equisingularity class of the pair of curves $f=0$ and $g=0$.

Consider the meromorphic function $\tilde g/\tilde  f:M\setminus E\to \bC\cup\{\infty\}$. 
We will check that this function extends analytically to the whole $M$ but a finite number of points. 
Let $\unit$ denotes any germ of a holomorphic function $u(x,y)$ such that $u(0,0)\neq0$.

First take $P\in E_i$ ($1\leq i\leq n$) which is not an intersection point with another component 
$E_j$ for~$1\leq j\leq m$. 
Then there exists a local analytical coordinate system $(x,y)$ centered at $P$ such that 
$E_i$ has an equation $x=0$. In these coordinates 
$\tilde f=\unit x^{a_i}$ and $\tilde g=\unit x^{b_i}$. 
We get $\tilde g/\tilde f=\unit x^{b_i-a_i}$. 

Now take the intersection point $P$ of $E_i$  with another component $E_j$. 
Choose  a local  analytical coordinate system $(x,y)$ centered at $P$ such that 
$E_i$ has an equation $x=0$ and $E_j$ has an equation $y=0$. 
In these coordinates 
$\tilde f=\unit x^{a_i}y^{a_j}$ and~
$\tilde g=\unit x^{b_i}y^{b_j}$. 
We get $\tilde g/\tilde f=\unit x^{b_i-a_i}y^{b_j-a_j}$.

Let $H$ be an analytic extension of $\tilde g/\tilde f$. Divide the set 
$\{\,E_1,\dots,E_m\,\}$ into three subsets 
$A_{+}=\{\,E_i: b_i-a_i>0\,\}$,
$A_{0}=\{\,E_i: b_i-a_i=0\,\}$ and 
$A_{-}=\{\,E_i: b_i-a_i<0\,\}$.
It follows from above description of $\tilde g/\tilde f$ near $E$ that 
$H$ is not defined only at intersection points of components from $A_{+}$ with 
components from $A_{-}$. 

Let $E_i\in A_{0}$. Consider the restriction $H|_{E_i}$ of the meromorphic function $H$ to $E_i$.
Then $P\in E_i$ is a zero of $H|_{E_i}$ if and only if $\{P\}= E_i\cap E_j$ for some  $E_j\in A_{+}$.  
Moreover $\ord_P H|_{E_i}=b_j-a_j$. Hence the topological degree of $H|_{E_i}$ is the 
number $d_i=\sum (b_j-a_j)$ where the sum runs over all $j$ such that $E_j\in A_{+}$ and the 
intersection $E_i\cap E_j$ is nonempty.  

\medskip
\noindent
Choose any complex number $t$ which is different from  
\begin{itemize}
\item  any critical value of meromorphic functions $H|_{E_i}$ where $E_i\in A_{0}$,
\item any value $H(P)$ where $P$ is the intersection point of some~ 
$E_i\in A_{0}$ with some~$E_j$, $j\neq i$. 
\end{itemize}

Let $\Gamma$ be the proper preimage of the curve $f-tg=0$. The curve $\Gamma$ has an equation $H=t$ at every 
point where $H$ is defined. Hence $\Gamma$ intersects transversally every $E_i\in A_0$ at $d_i$~points 
and none of these points belongs to  $\bigcup_{j\neq i}E_j$.  

Now we compute the equation of $\Gamma$ at points where $H$ is not defined. 
Take $E_i\in A_{+}$,  $E_j\in A_{-}$ with nonempty intersection and denote $P_{i,j}$ their 
intersection point. There exists a local analytical coordinate system $(x,y)$ centered at $P_{i,j}$ 
such that $E_i$ has an equation $x=0$ and $E_j$ has an equation $y=0$. 
 In these coordinates $\tilde f-t\tilde g=\unit x^{a_i}y^{a_j}-t\unit x^{b_i}y^{b_j}=x^{a_i}y^{b_j}(\unit y^{a_j-b_j}-t \unit x^{b_i-a_i})$ 
hence $\Gamma$ has the equation $\unit y^{a_j-b_j}-t x^{b_i-a_i}=0$. 

We  want to resolve singularities of the curve $\Gamma$ to obtain a good (not necessarily minimal) 
resolution of singularities of $f-tg=0$. 
The functions $h_{i,j}(x,y)=\unit y^{a_j-b_j}-t x^{b_i-a_i}$
are nondegenerate with Newton diagrams $\Teisssr{b_i-a_i}{a_j-b_j}{8}{4}$. 
Hence by Theorem~4.3 of~\cite{Oka} there exists a canonical toric resolution of $h_{i,j}(x,y)=0$ at the 
origin, that is the resolution of $\Gamma$ at $P_{i,j}$,  which depends only on the Newton diagram 
of $h_{i,j}$.  Applying such a toric resolution at every point $P_{i,j}$ described above we obtain 
a good resolution of $f-tg=0$. Moreover the total dual resolution graph of this resolution depends 
only on the total dual resolution graph of $R$ and on the numbers $a_i$ and $b_i$ for $i=1,\dots,m$. 
Since the total dual resolution graph of the plane curve singularity determines its equisingularity class, 
the proof is finished.
\end{proof}

\section{Proof of the main result}
For all holomorphic functions $h_i$  ($i=1,2$) defined in the 
neighborhood of the origin of $\bC^2$ we will denote $i_0(h_1,h_2)$ the intersection multiplicity 
of curves $h_1=0$ and $h_2=0$ at zero and $\mu_0(h_1)$ the Milnor number of 
the curve $h_1=0$ at zero. For every Newton diagram $\Delta$ and for every 
$\vec v =(v_1,v_2)$, where $v_1>0$, $v_2>0$  we define 
$$l(\vec  v,\Delta)=\min\{\,v_1i+v_2j: (i,j)\in\Delta\,\}.$$ 

\begin{Lemma}\label{L:I}
Let $\vec v =(m,n)$, where $n$, $m$ are co-prime positive integers.
Then for generic $t\in\bC$
$$ l\bigl(\vec v,\Nj(f,g)\bigr)= \mu_0(f^n-tg^m) -i_0(f,g)[(m-1)(n-1)-1] -1.
$$
\end{Lemma}

\begin{proof}
Let $D=0$, where $D(x,y)=\sum c_{ij}x^iy^j$  be the equation of the discriminant 
curve of $\phi=(f,g):(\bC^2,0)\to(\bC^2,0)$.  Take a curve $x^n-ty^m=0$. 

\medskip
\textbf{Claim.} 
For generic $t\in\bC$ there is 
$i_0(x^n-ty^m,D)=l\bigl(\vec v,\Nj(f,g)\bigr)$.

\smallskip

Let $\tau=\sqrt[n]{t}$. Then $x=\tau s^m$, $y=s^n$ is a parametrization of  the branch $x^n-ty^m=0$.
By the classical formula for the intersection multiplicity
$$i_0(x^n-ty^m,D)=\ord_s D(\tau s^m,s^n) = 
\ord_s \sum c_{ij}\tau^i s^{mi+nj}=l\bigl(\vec v,\Nj(f,g)\bigr)$$
provided $\tau$ is sufficiently general
so that the  sum $\sum_{mi+nj=l(\vec v,\Nj(f,g))} c_{ij}\tau^i$ is nonzero. 
The Claim is proved.

\medskip
The pull-back of a curve $x^n-ty^m=0$ by $\phi$ has an equation $f^n-tg^m=0$. 
Thus by Theorem~3.2 of \cite{Casas} there is
$$ \mu_0(f^n-tg^m) -1 =    i_0(f,g)[\mu_0(x^n-ty^m)-1]+ i_0(x^n-ty^m,D)
$$
which gives the Lemma because $\mu_0(x^n-ty^m)=(m-1)(n-1)$. 
\end{proof}

\medskip
\begin{proof}[Proof of Theorem~\ref{Th:2}]
It follows from Lemma~\ref{L:I} and Theorem~\ref{Th:1} that for every vector
$\vec v=(m,n)$,  where $m$, $n$ are co-prime positive integers,
the number $l\bigl(\vec v,\Nj(f,g)\bigr)$ depends only on the 
equisingularity class of the pair $f=0$ and $g=0$. 
This proves the Theorem.
\end{proof}

\end{document}